\documentclass[12pt]{amsart}
\usepackage[latin1]{inputenc}
\usepackage{amsbsy}
\usepackage{amscd} 
\usepackage{amsfonts}
\usepackage{amsgen} 
\usepackage{amsmath}
\usepackage{amsopn} 
\usepackage{amssymb}
\usepackage{amstext}
\usepackage{amsthm} 
\usepackage{amsxtra}
\usepackage{rotating}
\usepackage{MnSymbol}
\usepackage{picture}

\theoremstyle{plain} 

\newtheorem{theoremcounter}{theoremcounter}[section]

\newtheorem{corollary}[theoremcounter]{Corollary}
\newtheorem{lemma}[theoremcounter]{Lemma}
\newtheorem{proposition}[theoremcounter]{Proposition}
\newtheorem{theorem}[theoremcounter]{Theorem}

\theoremstyle{definition}
\newtheorem{definition}[theoremcounter]{Definition}
\newtheorem{example}[theoremcounter]{Example}
\newtheorem{remark}[theoremcounter]{Remark}

\numberwithin{equation}{section}

\renewcommand{\theta}{\vartheta}
\renewcommand{\phi}{\varphi}
\renewcommand{\epsilon}{\varepsilon}
\renewcommand{\subset}{\subseteq}

\newcommand{\N}{\mathbb N}
\newcommand{\NN}{\mathbb N}
\newcommand{\Z}{\mathbb Z}
\newcommand{\ZZ}{\mathbb Z}

\newcommand{\C}{\mathbb C}

\newcommand{\CC}{\mathcal C}
\newcommand{\cC}{\mathcal C}

\newcommand{\AKat}{A_\CC(n)}

\newcommand{\ot}{\otimes}

\newcommand{\idpart}{|}
\newcommand{\paarpart}{\sqcap}
\newcommand{\baarpartbaustein}{\rotatebox{180}{$\sqcap$}}
\newcommand{\baarpart}{
\mathrel{\vcenter{\offinterlineskip \hbox{$\baarpartbaustein$}}}}

\newcommand{\singleton}{\uparrow}

\newcommand{\vierpart}{
\mathrel{\offinterlineskip
\hskip0ex\hbox{$\sqcap$}\hskip -.35ex\hbox{$\sqcap$} \hskip -0.35ex\hbox{$\sqcap$}}}

\newcommand{\crosspart}{
\mathrel{\offinterlineskip
\hskip 0.1ex \hbox{$/$}\hskip -.95ex\hbox{$\backslash$}}}

\newcommand{\primarypart}{
\mathrel{\vcenter{\offinterlineskip
\hbox{$\baarpart$} \vskip -1.3ex \hbox{\hskip1.3ex$/$\hskip-1.2ex$-$} \vskip -1.2ex \hbox{\hskip2.2ex $\paarpart$}}}}

\newcommand{\midmid}{
\mathrel{\vcenter{\offinterlineskip
\vskip -0.2ex \hbox{$\shortmid$} \vskip -0.75ex \hbox{$\shortmid$}}}}

\newcommand{\halflibpart}{
\mathrel{\offinterlineskip
\hbox{$\bigtimes$}\hskip -1.55ex \hbox{$\midmid$} \hskip 1ex}}

\usepackage{calc}
\newcounter{PartitionDepth}
\newcounter{PartitionLength}

\newcommand{\partii}[3]{
 \begin{picture}(#3,#1)
 \setcounter{PartitionLength}{#3-#2}
 \setcounter{PartitionDepth}{-1-#1}
 \put(#2,\thePartitionDepth){\line(0,1){#1}}     
 \put(#3,\thePartitionDepth){\line(0,1){#1}}
 \put(#2,\thePartitionDepth){\line(1,0){\thePartitionLength}}
 \end{picture}}

\newcommand{\upparti}[2]{
 \begin{picture}(#2,#1)
 \setcounter{PartitionDepth}{#1}
 \put(#2,0){\line(0,1){#1}}
 \end{picture}}
\newcommand{\uppartii}[3]{
 \begin{picture}(#3,#1)
 \setcounter{PartitionLength}{#3-#2}
 \setcounter{PartitionDepth}{#1}
 \put(#2,0){\line(0,1){#1}}     
 \put(#3,0){\line(0,1){#1}}
 \put(#2,\thePartitionDepth){\line(1,0){\thePartitionLength}}
 \end{picture}}

\DeclareMathOperator{\Hom}{Hom}
\DeclareMathOperator{\End}{End}
\DeclareMathOperator{\mini}{min}

\begin{document}
\title{The combinatorics of an algebraic class of easy quantum  groups}
\author[Sven Raum and Moritz Weber]{Sven Raum$^{(1)}$ and Moritz Weber}
  \setcounter{footnote}{1}
  \footnotetext{Supported by KU Leuven BOF research grant OT/08/032 and by ANR Grant NEUMANN}
\address{Sven Raum, \'Ecole Normale Sup\'erieure de Lyon, Unit\'e de Math\'ematiques Pures et Appliqu\'ees, UMR CNRS 5669,  69364 Lyon Cedex 07, France}
\email{sven.raum@ens-lyon.fr}
\address{Moritz Weber, Saarland University, Fachbereich Mathematik, Postfach 151150,
66041 Saarbr\"ucken, Germany}
\email{weber@math.uni-sb.de}
\date{\today}

\begin{abstract}
Easy quantum groups are compact matrix quantum groups, whose intertwiner spaces are given by the combinatorics of categories of partitions. This class contains the symmetric group $S_n$ and the orthogonal group $O_n$ as well as Wang's quantum permutation group $S_n^+$ and his free orthogonal quantum group $O_n^+$.
In this article, we study a particular class of categories of partitions to each of which we assign a subgroup of the infinite free product of the cyclic group of order two. This is an important step in the classification of all easy quantum groups and we deduce that there are uncountably many of them. 
We focus on the combinatorial aspects of this assignment, complementing the quantum algebraic point of view presented in another article.
\end{abstract}

\maketitle

\section*{Introduction}
The main objects of this article are \emph{categories of partitions}. Given $k+l$ points, $k$ of which are placed on an upper row and $l$ on a lower one, we can partition this set into a disjoint union of subsets. To describe this partition, we draw lines connecting the points belonging to the same subset. The collection of all such partitions is denoted by $P(k,l)$ and $P$ denotes the union of all $P(k,l)$. A category of partitions $\CC$ is a subset of $P$ which is closed under certain operations (see Section \ref{SectCatOfPart}) corresponding to operations of a tensor category.

Categories of partitions were introduced by Banica and Speicher in 2009 \cite{BanicaSpeicher09} in order to study certain compact matrix quantum groups, so called \emph{easy quantum groups}. Although our article is mainly of combinatorial nature, let us very briefly sketch the idea of easy quantum groups and refer to \cite{BanicaSpeicher09}, \cite{Weber13}, \cite{RaumWeberSemiDirect} or  \cite{RaumWeberVolleKlassif} for details.

In operator algebras, the study of noncommutative spaces requires a more general notion of symmetries. 
In 1987, Woronowicz \cite{Woronowicz87} gave a definition of \emph{compact matrix quantum groups}, generalizing compact Lie groups $G\subset M_n(\C)$. The idea is to pass from $G$ to the algebra $C(G)$ of continuous functions and to study noncommutative variants $A$ of it. Here, the group law $\mu:G\times G\to G$ is dualized to a $^*$-homomorphism $\Delta:A\to A\otimes A$. Typically compact matrix quantum groups are constructed as universal C$^*$-algebras generated by elements $u_{ij}$, $1\leq i,j\leq n$ such that the matrix $u$ formed by the $u_{ij}$'s satisfies some relations. In the case where $A=C(G)$ is the algebra of functions over a group $G\subset M_n(\C)$, the C$^*$-algebra is commutative and the elements $u_{ij}$ are the functions in $C(G)$ evaluating the entries of the matrices in the group $G$.

By a Tannaka-Krein result of Woronowicz \cite{Woronowicz88}, a category of partitions gives rise to a tensor category and hence to intertwiner spaces of a compact matrix quantum group. If a compact matrix quantum group arises in this way, it is called an easy quantum group \cite{BanicaSpeicher09}. Thus, Banica and Speicher's easy quantum groups carry a very intrinsic combinatorial data. This interplay between combinatorial and operator algebraic aspects turned out to be quite fruitful (see for instance \cite{RaumWeberVolleKlassif} for an overview of results on easy quantum groups and their applications). For instance they give rise to appropriate symmetries in Voiculescu's free probability (see K\"ostler and Speicher's free de Finetti result \cite{KoestlerSpeicher09} and more references in \cite{RaumWeberVolleKlassif}).
The class of easy quantum groups interpolates the symmetric group $S_n$ and Wang's \cite{Wang95} free orthogonal quantum  group $O_n^+$, amongst others containing the quantum permutation group $S_n^+$ introduced by Wang \cite{Wang98}.

In the present article, we contribute to the classification of easy quantum groups and to the classification of categories of partitions respectively.
This project has been started by Banica and Speicher in \cite{BanicaSpeicher09}, and it has been continued by Banica,  Curran, Speicher \cite{BanicaCurranSpeicher10} and the second author \cite{Weber13}.
It turned out that in the case of nonhyperoctahedral categories of partitions, a complete list of 13 categories is available \cite{BanicaCurranSpeicher10}, \cite{Weber13}. In the hyperoctahedral case however (categories containing the four block partition $\vierpart$ but not the double singleton $\singleton\otimes\singleton$), the situation was quite unclear (see eg. \cite{BanicaCurranSpeicher10}).

We study the class of hyperoctahedral categories containing the pair positioner partition $\primarypart$ -- the \emph{group-theoretical categories of partitions}. They correspond to quantum subgroups of the free hyperoctahedral quantum group $H_n^+$ of Banica, Bichon and Collins \cite{BanicaBichonCollins07} satisfying the additional requirement on the generators:
\[u_{ij}^2u_{kl}=u_{kl}u_{ij}^2\]
Surprisingly, these categories carry a quite algebraic structure: Since the partition $\primarypart$ allows us to shift pairs of points to arbitrary positions, the only important information about such a category is contained in its partitions that look like reduced words in the free product $\Z_2^{*\infty}$ of the cyclic group of order two. We hence find a map $F$ assigning to each group-theoretical category of partitions a subgroup of the infinite free product $\Z_2^{*\infty}$. We are able to investigate its image, so that we obtain a lattice isomorphism between group-theoretical categories of partitions and a specific class of normal subgroups of $\Z_2^{*\infty}$ (see our main theorem, Theorem \ref{thm:F-is-1-1}).

This article is based on the preprint \cite{RaumWeberPreprint} which is not intended for publication. In the present work, we focus on the clear exposition of the combinatorial material presented in the preprint.
Hence, we aim to clarify the existence of the one-to-one assignment $F$ from a combinatorial point of view.
In another article \cite{RaumWeberSemiDirect} we focus on the quantum algebraic point of view and the assignment is derived via diagonal subgroups of the quantum groups (even in a much more general setting).
Let us remark that the second article \cite{RaumWeberSemiDirect} extends the insight obtained in the preprint \cite{RaumWeberPreprint} by far. However, in \cite{RaumWeberSemiDirect}, we develop an approach different from the one of the preprint, using less combinatorics. In the present article we present the purely combinatorial approach.

  Consequences of the bijectivity of the assignment $F$ may be found in \cite{RaumWeberSemiDirect}. Nevertheless, some of them could also be directly derived from the results of the present article. For instance, we infer that there is an embedding of the class of all varieties of groups into categories of partitions and hence there are uncountably many categories or rather uncountably many easy quantum groups. Moreover, the problem of classifying all group-theoretical easy quantum groups is transferred to a classification problem in group theory which includes the problem of finding all varieties of groups. This clearly destroys the hope to give an explicit list of all easy quantum groups. Nevertheless, our assignment $F$ clarifies a lot about the structure of group-theoretical easy quantum groups. 
The case of non group-theoretical hyperoctahedral quantum groups is treated in \cite{RaumWeberVolleKlassif}, completing the classification of all (orthogonal) easy quantum groups.

In Section \ref{SectCatOfPart} we give an introduction into the main combinatorial object,  categories of partitions. In Section \ref{SectSingleLeg} we  work out a normal form for partitions in group-theoretical categories of partitions. This enables us to assign certain groups to the categories. This assignment $F$ is constructed and investigated in the subsequent section. A short introduction into easy quantum groups is given in Section \ref{SectConsequences}, where we summarize consequences of our classification. We end this article with an analysis of the C$^*$-algebraic relations related to group-theoretical categories of partitions.

\section*{Acknowledgements}
The first author thanks Roland Speicher for inviting him to Saarbr\"ucken, where this work was initiated in June 2012.  Both authors thank Stephen Curran for pointing out an error at an early stage of our work on the preprint \cite{RaumWeberPreprint}, of which this article is a continuation.  We are very grateful to Teodor Banica and Adam Skalski for useful discussions on the same preprint.

\section{Categories of partitions}\label{SectCatOfPart}

The starting point of this article are categories of partitions. Let us give a brief introduction into these combinatorial objects. See also \cite[Section 1]{RaumWeberVolleKlassif}.

A \emph{partition} $p$ is given by  $k\in\N_0$ upper and $l\in\N_0$ lower points, which may be connected by some lines. Thus, the set of $k+l$ points is partitioned into several subsets called \emph{blocks}. The set of all such partitions is denoted by $P(k,l)$ and we write $P$ for the collection of all $P(k,l)$ for $k,l\in\N_0$. As an example, consider the following partition where we label the upper and lower points by the numbers $1,2,3$ and $1',2',3'$, respectively. It is given by a four block consisting of the points $1,2,2',3'$ and a block of size two consisting of $3$ and $1'$.

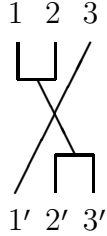
\setlength{\unitlength}{0.5cm}
\begin{center}
  \begin{picture}(5,6)
    \thicklines
    \put(0,6){\partii{1}{1}{2}}
    \put(1.2,1){\line(1,2){2}}
    \put(1.8,4){\line(1,-2){1}}
    \put(0,1){\uppartii{1}{2}{3}}
    \put(1,5.5){1}
    \put(2,5.5){2}
    \put(3,5.5){3}
    \put(1,0){$1'$}
    \put(2,0){$2'$}
    \put(3,0){$3'$}
  \end{picture}
\end{center}

This partition is called the \emph{pair positioner partition}, and it is denoted by $\primarypart\;\in P(3,3)$. It will play a crucial role in this article. Further examples are the \emph{singleton partition} $\singleton\;\in P(0,1)$ consisting of a single lower point; the \emph{double singleton partition} $\singleton\otimes\singleton\;\in P(0,2)$ on two non-connected lower points; the \emph{pair partition} $\paarpart\;\in P(0,2)$ on two connected lower points; the  \emph{identity partition} $\idpart\;\in P(1,1)$ connecting one upper with one lower point; the \emph{four block partition} $\vierpart\;\in P(0,4)$ connecting four lower points; the  \emph{crossing partition} $\crosspart\;\in P(2,2)$ on two crossing pairs $\{1,2'\}$ and $\{2,1'\}$; the \emph{half-liberating partition} $\halflibpart\;\in P(3,3)$ given by the blocks  $\{1,3'\}$, $\{2,2'\}$ and $\{3,1'\}$; and the \emph{s-mixing partition} $h_s\in P(0,2s)$ for $s\in \N$ given by two blocks connecting $s$ odd points and $s$ even ones, respectively.
If $k = l = 0$, then $P(0, 0)$ consists only of the \emph{empty partition}  $\emptyset$. 

We view partitions also as words in the following way. Let $p\in P(0,l)$ be a partition with $k$ blocks. We label all points of the first block by a letter $a_1$, those of the second by $a_2$ and so on. We obtain a word in the letters $a_1,\ldots,a_k$. For instance, the \linebreak $s$-mixing partition $h_s$ is given by the word $h_s=abab\ldots ab$ of length $2s$.
Conversely, if we are given a word of length $l$ in $k$ letters, we obtain a partition $p\in P(0,l)$ by connecting all equal letters by lines. 

In \cite{BanicaSpeicher09} Banica and Speicher introduced several operations on the set $P$ of partitions.
\begin{itemize}
\item  The \emph{tensor product} of two partitions $p\in P(k, l)$ and $q\in P(k', l')$ is the partition $p\otimes q\in P(k+k', l+l')$ obtained by \emph{horizontal concatenation} (writing $p$ and $q$ next to each other), i.e. the first $k$ of the $k+k'$ upper points are connected by $p$ to some of the first $l$ of the $l+l'$ lower points, whereas $q$ connects the remaining $k'$ upper points with the remaining $l'$ lower points.
\item The \emph{composition} of two partitions $p\in P(k, l)$ and $q\in P(l, m)$ is the partition  $qp\in P(k, m)$ obtained by \emph{vertical concatenation}. Connect $k$ upper points by $p$ to $l$ middle points and then continue the lines by $q$ to $m$ lower points. This yields a partition, connecting $k$ upper points with $m$ lower points. The $l$ middle points and lines not connected to the remaining $k+m$ points are removed.
\item The \emph{involution} of a partition $p\in P(k, l)$ is the partition $p^{*}\in P(l, k)$ obtained by turning $p$ upside down.
\item We also have a \emph{rotation} on partitions. Let $p\in P(k, l)$ be a partition. Shifting the very left upper point to the left of the lower points --  it then still belongs to the same block as before -- gives rise to a partition in $P(k-1, l+1)$, called a \emph{rotated version} of $p$. We may also shift the left lower point to the upper line, and we may rotate on the right-hand side of the two lines of points as well. In particular, for a partition $p\in P(0, l)$, we may rotate the very left point to the very right and vice-versa.
\end{itemize}

The above operations are called the \emph{category operations}. 
A collection $\CC$ of subsets $\CC(k, l)\subseteq P(k, l)$ (for every $k, l\in\N_{0}$) is a \emph{category of partitions} if it is closed under the category operations and if the identity partition $\idpart\in P(1, 1)$ is in $\CC$ (by rotation, we then also have $\paarpart\in\CC$).
Note that rotation may be deduced from the other axioms using $\paarpart\in\CC$, see \cite{BanicaSpeicher09}.
Examples of categories of partitions include $P$, $P_{2}$ (the set of all \emph{pair partitions}, i.e. all blocks have size two), $NC$ (all \emph{noncrossing partitions}, i.e. partitions which may be drawn such that the lines do not cross) and $NC_{2}$. 
We often restrict our considerations to partitions without upper points when working with categories of partitions. This is possible, since $p\in P(k,l)$ is in $\CC$ if and only if a rotated version $p'\in P(0,k+l)$ is in $\CC$.
We write $\CC=\langle p_1,\ldots, p_n\rangle$, if $\CC$ is the smallest category of partitions containing the partitions $p_1,\ldots, p_n$. We say that $\CC$ is \emph{generated} by $p_1,\ldots,p_n$.
For example, we have $P=\langle\crosspart,\singleton,\vierpart\rangle$ whereas $NC=\langle\singleton,\vierpart\rangle$ (see \cite{Weber13}).

The terms ``\emph{applying the pair partition} to a partition $p\in P(k,l)$'' or ``\emph{using the pair partition}'' refer to the composition of $p$ with $\idpart^{\otimes\alpha}\otimes\baarpart\otimes\idpart^{\otimes\beta}$ for suitable $\alpha$ and $\beta$ (possibly iteratively). Note that this  operation can be done inside a category, i.e. if $p$ is in $\CC$, then also the partition resulting from the above procedure lies in $\CC$.

\section{Group-theoretical categories of partitions} \label{SectSingleLeg}

A category of partitions is called \emph{hyperoctahedral} if it contains the four block partition $\vierpart$ but not the double singleton $\singleton\otimes\singleton$.
In the classification of categories of partitions (\cite{BanicaSpeicher09}, \cite{BanicaCurranSpeicher10}, \cite{Weber13}), it turned out that this case is complicated. In the \emph{nonhyperoctahedral} case in turn, there are exactly 13 categories (\cite[Th. 3.12]{Weber13}). 
In this article we study a subclass of hyperoctahedral categories.

\begin{definition}
 A category $\CC$ of partitions is called \emph{group-theoretical}, if it contains the pair positioner partition $\primarypart$.
\end{definition}

We mainly intend to shed some light on hyperoctahedral categories, thus we want to understand the class of hyperoctahedral group-theoretical categories of partitions. At the same time, there are only two nonhyperoctahedral categories amongst the group-theoretical categories, namely $\langle\crosspart, \vierpart,\singleton\otimes\singleton\rangle$ and $\langle\crosspart, \vierpart,\singleton\rangle$. Hence, it is basically equivalent to study \emph{all} group-theoretical categories instead of only hyperoctahedral ones.
For the non group-theoretical hyperoctahedral case, we refer to \cite{RaumWeberVolleKlassif}.

The crucial feature in group-theoretical categories is the fact that it suffices to  study partitions in a ``normal'' form. Let us prepare the necessary tools for doing so.
Let $p\in P(0,l)$ be a partition with $m$ blocks. We view it as a word in the letters $a_1,\ldots,a_k$, where $k_j\in\N$:
\[p=a_{i(1)}^{k_1}a_{i(2)}^{k_2}\ldots a_{i(n)}^{k_n}\]

\begin{definition}
 A partition $p\in P(0,l)$ is in \emph{single leg form} if $p$ is -- as a word -- of the form
\[p=a_{i(1)}a_{i(2)}\ldots a_{i(n)}\;,\]
such that $a_{i(j)}\neq a_ {i(j+1)}$ for $j=1,\ldots,n-1$.
In other words, a partition is in single leg form if no two consecutive points belong to the same block.
If $\mathcal C$ is a set of partitions, we denote by $\mathcal C_{sl}$ the set of all partitions $p\in\mathcal C$ in single leg form.
\end{definition}

We now derive how group-theoretical categories of partitions are determined by their partitions in single leg form.
Let us introduce the following 
partitions $k_l\in P(l+2,l+2)$ for $l\in \N$. They are given by a four block  $\{1,1',l+2,(l+2)'\}$ and pairs $\{2,2'\}, \{3,3'\},\ldots,\{l+1,(l+1)'\}$. The following picture illustrates the partition $k_l$ -- note that the waved line from $1'$ to $l+2$ is \emph{not} connected to the lines from $2$ to $2'$, from $3$ to $3'$ etc.
\setlength{\unitlength}{0.5cm}
\begin{center}
  \begin{picture}(14,5)
    \put(0.5,2.5){$k_l\;\;=$}
    \put(1.9,1.5){\upparti{2}{1}}
    \put(1.9,1.5){\upparti{2}{2}}
    \put(1.9,1.5){\upparti{2}{3}}
    \put(7,2){$\ldots$}
    \put(1.9,1.5){\upparti{2}{8}}
    \put(1.9,1.5){\upparti{2}{9}}
    \put(3,0){$1'$}
    \put(4,0){$2'$}
    \put(5,0){$3'$}
    \put(7,0){$\ldots$}
    \put(8.4,0){$(l+1)'$}
    \put(11,0){$(l+2)'$}
    \put(3,4){$1$}
    \put(4,4){$2$}
    \put(5,4){$3$}
    \put(7,4){$\ldots$}
    \put(9.1,4){$l+1$}
    \put(10.9,4){$l+2$}
    
    \put(7.2,1.5){\oval(8,2)[tl]}
    \put(7.15,3.5){\oval(8,2)[br]}
  \end{picture}
\end{center}

We check that $k_1$ is in a category $\mathcal C$ if and only if $k_l$ is in $\cC$ for all $l\in\N$. This follows from the construction
$k_{l+1}=(\idpart^{\otimes l+1}\otimes\baarpart\otimes\idpart^{\otimes 2})(k_l\otimes k_1)(\idpart^{\otimes l+1}\otimes \paarpart\otimes\idpart^{\otimes 2})$. 
The pair positioner partition $\primarypart$ is a rotated version of $k_1$, thus $\primarypart\in\mathcal C$ if and only if $k_1\in\mathcal C$.

\begin{lemma}\label{LemBlockVerbinden}
 Let $\CC$ be a group-theoretical category of partitions  and let $p\in\CC$.
Then, we can connect arbitrary blocks of $p$ inside of $\cC$, i.e. the partition $p'$ obtained from $p$ by combining two arbitrary blocks of $p$ is again in $\cC$.
\end{lemma}
\begin{proof}
Assume $p\in P(0,l)$ and compose $p$ with a suitable $\idpart^{\otimes \alpha}\otimes k_l\otimes\idpart^{\otimes \beta}$, which is in $\CC$ by the above considerations.
\end{proof}

\begin{remark} \label{RemKL}
In \cite[Lemma 4.2]{BanicaCurranSpeicher10} the  partitions $k_l\in P(l+2,l+2)$ were used to define the \emph{higher hyperoctahedral series} $H_n^{[s]}$. The category  $\langle\primarypart\rangle$ corresponds to $H_n^{[\infty]}$ of \cite{BanicaCurranSpeicher10}.
\end{remark}

\begin{lemma}\label{LemThreeRow}\label{LemReduzieren}
 Let $\cC$ be a group-theoretical category of partitions and let $p\in P(0,l)$ be a partition, seen as the word $p=a_{i(1)}^{k_1}a_{i(2)}^{k_2}\ldots a_{i(n)}^{k_n}$.
 We put $k_j':= \begin{cases} 1 &\textnormal{if $k_j$ is odd}\\ 0 &\textnormal{if $k_j$ is even}\end{cases}$, and $p':=a_{i(1)}^{k_1'}a_{i(2)}^{k_2'}\ldots a_{i(n)}^{k_n'}$. It is possible that $p'=\emptyset$.
 Then $p\in\mathcal C$ if and only if $p'\in\mathcal C$. 
\end{lemma}
\begin{proof}
If $p\in\cC$, then $p'\in \cC$ using the pair partition. For the converse, insert a pair partition $\paarpart$ (via composition with $\idpart^{\otimes\alpha}\otimes\paarpart\otimes\idpart^{\otimes\beta}$) next to $a_{i(j)}^{k_j'}$ in $p'$ and use Lemma~\ref{LemBlockVerbinden}. This increases $k_j'$ by two -- even if $k_j'$ is zero. Iteratively we obtain $p\in\CC$.
\end{proof}

The above lemma can be extended to arbitrary partitions $p\in P(k,l)$, by rotation.
If $p\in P(0,l)$ is a partition seen as the word $p=a_{i(1)}^{k_1}a_{i(2)}^{k_2}\ldots a_{i(n)}^{k_n}$ then the partition $p'=a_{i(1)}^{k'_1}a_{i(2)}^{k'_2}\ldots a_{i(n)}^{k'_n}$ of Lemma \ref{LemReduzieren} is not necessarily in single leg form -- even if $a_{i(j)}\neq a_{i(j+1)}$ is satisfied. For instance, $p=ab^2acacaca$ yields $p'=a^2cacaca$. However, a finite iteration of the procedure as in Lemma \ref{LemReduzieren} yields a partition $q$ in single leg form. This partition $q$ (possibly the empty partition $\emptyset\in P(0,0)$) is called the \emph{single leg partition associated to} $p$ or the \emph{single leg version of} $p$. Note that every partition has a unique single leg version. The converse is not true, different partitions may have the same associated single leg partitions. Using Lemma \ref{LemReduzieren}, we prove the main result of this section.

\begin{proposition}\label{LemmaX}
 Let $\cC$ be a group-theoretical category of partitions. Then, a partition is in $\cC$ if and only if its single leg version is in $\cC$.
\end{proposition}
\begin{proof}
 Iteration of Lemma \ref{LemReduzieren}.
\end{proof}

The set $\mathcal C_{sl}$ turns out to be a complete invariant for group-theoretical categories.

\begin{corollary}\label{PropCSL}
 Let $\mathcal C$ and $\mathcal D$ be group-theoretical categories.
\begin{itemize}
 \item[(a)] The category $\langle\mathcal C_{sl},\primarypart\rangle$ coincides with $\mathcal C$.
 \item[(b)] We have $\mathcal C_{sl}=\mathcal D_{sl}$ if and only if $\mathcal C=\mathcal D$.
 \item[(c)] We have $\mathcal C_{sl}\subseteq\mathcal D_{sl}$ if and only if $\mathcal C\subseteq\mathcal D$.
\end{itemize}
\end{corollary}
\begin{proof}
 (a) We have  $\langle\mathcal C_{sl},\primarypart\rangle\subset\mathcal C$. On the other hand, if $p\in \cC$, we consider its associated simplified partition $p'$ which is in $\cC_{sl}$ by Proposition \ref{LemmaX}. Thus, $p'$ is in $\langle\mathcal C_{sl},\primarypart\rangle$. Again by Proposition \ref{LemmaX}, we infer $p\in\langle\mathcal C_{sl},\primarypart\rangle$.
 
 (b) and (c) follow directly from (a).
\end{proof}

As a consequence, we can always choose the generators of a group-theoretical category to be in single leg form. Furthermore, we may study group-theoretical categories by analyzing the structure of $\mathcal C_{sl}$. This is the crucial observation from the combinatorial point of view: The partitions in single leg form completely determine a group-theoretical category -- and they are ``group-theoretical'' indeed, as may be seen in the next section.

\section{Associating groups to group-theoretical categories of partitions}\label{SectGroups}

Denote by $a_1, a_2 ,\ldots$ the generators of the infinite free product $\ZZ_2^{* \infty}$ of the cyclic group of order two.
Let $p\in P(k,m)$ be a partition and choose a labelling $l = (a_{i(1)}, a_{i(2)}, \dotsc, a_{i(k+m)})$ of the points of $p$.
We start our labelling with the lower left point of $p$ by $a_{i(1)}$ and then go around counterclockwise. We require that all points of the same block are labelled by the same letter. Furthermore, mutually different blocks are labelled by mutually different letters.
Rotating the $k$ upper points of $p$ to the right of the $l$ lower points, we obtain an element $w(p,l)$ in the free product $\ZZ_2^{*\infty}$.

\begin{remark}\label{RemSL}
The labelling of the partition $p$ yields a not necessarily reduced word in $\ZZ_2^{*\infty}$ unless $p$ is in single leg form.
Furthermore, $w(p,l)$ only depends on the single leg  partition $p'$ associated to $p$, i.e. there is a labelling $l'$ such that $w(p,l)=w(p',l')$.  In other words, passing from $\CC$ to $\CC_{sl}$ corresponds exactly to the effect of reducing the words in $\ZZ_2^{*\infty}$ after the labelling. We thus have $F(\CC)=F(\CC_{sl})$.
\end{remark}

\begin{definition}
  Let $\cC$ be a group-theoretical category of partitions.  We denote by $F(\cC)$ the subset of $\ZZ_2^{* \infty}$ formed by all elements $w(p,l)$ where $p \in \cC$ and $l$ runs through all possible labellings of the points of $p$.
\end{definition}

Next, we study the structure of $F(\CC)$ for group-theoretical categories $\CC$. The category operations give rise to operations in  $\ZZ_2^{* \infty}$.

\begin{lemma}\label{LemFCGruppe}
 Let $\cC$ be a group-theoretical category of partitions. Then, $F(\cC) \subset \ZZ_2^{* \infty}$ is a normal subgroup of $\ZZ_2^{* \infty}$.
\end{lemma}
\begin{proof} 
By Remark \ref{RemSL}, we may always assume that elements $g=w(p,l)$ in $F(\cC)$ are represented by partitions $p\in\cC$ in single leg form with no upper points.

The set $F(\cC)$ is closed under taking inverses. Indeed, note that the inverse of a word in $\Z_2^{*\infty}$ is simply the reversed word, because $a_i^2=e$. Let $g=w(p,l)$ for some partition $p\in\cC$ in single leg form without upper points and some labelling $l$. Thus, $w(p,l)$ is given by labelling the partition $p$ from left to right. As $p$ is in $\CC$, the partition $\bar p$ obtained from vertical reflection is also in $\CC$. Indeed, simply rotate the points of $p^*$ (they are all on the upper line) to the lower line -- and we obtain $\bar p\in\CC$. Now, labelling $\bar p$ by the labelling $l$ from right to left (we denote this labelling by $\bar l$), we infer that the reversed word $g^{-1}=w(\bar p,\bar l)$ is in $F(\cC)$.

In order to prove that $F(\cC)$ is closed under taking products, let $g=w(p,l)$ and $h=w(q,l')$ for some partitions $p,q\in \cC$ in single leg form without upper points, and some labellings $l = (a_{i(1)}, \dotsc, a_{i(n)})$ and $l' = (a_{j(1)}, \dotsc a_{j(m)})$.  If all letters of $l$ and $l'$ are pairwise different, then $gh=w(p\otimes q,ll')$, where $ll'$ is the labelling $ll' = (a_{i(1)}, \dotsc, a_{i(n)},a_{j(1)}, \dotsc a_{j(m)})$. Otherwise, 
if $a_{i(\alpha)}=a_{j(\beta)}$ for some $\alpha$ and $\beta$, we connect the corresponding blocks of $p$ and $q$, using Lemma \ref{LemBlockVerbinden}. Iterating this procedure yields a  partition $r$ in $\cC$ and $gh=w(r,l'')$ with the labelling $l''$ obtained from $l$ and $l'$.

Finally, $F(\cC)$ is closed under conjugation with a letter $a_k$. Let $e \neq g = w(p,l)$ be an element in $F(\cC)$ such that $p$ is in single leg form with no upper points.
 If the letter $a_k$ does not appear in the labelling $l = (a_{i(1)}, \dotsc, a_{i(n)})$ of $p$, we consider the partition
\begin{center}
\begin{picture}(4,3)
  \put(0,0.75){$p' =$}
  \put(1,0){$\uppartii{2}{1}{4}$}
  \put(3.25,0.5){$p$}

  \put(5.25,0.25){,}
\end{picture}
\end{center}
i.e. the partition obtained from composing $\paarpart$ with $\idpart\otimes p\otimes\idpart$. Labelling this partition with $l'=(a_k,a_{i(1)}, \dotsc, a_{i(n)}, a_k)$ yields $a_k g a_k = w(p',l')$ in $F(\cC)$. In the case that the letter $a_k$ appears in the labelling $l$, we connect the corresponding block of $p$ with the outer pair, using Lemma \ref{LemBlockVerbinden}. The resulting partition $r$ yields $a_kga_k=w(r,l'')$ for some labelling $l''$.
\end{proof}

We  give a description of the image of $F$ in terms of subgroups of $\ZZ_2^{*\infty}$ that are invariant under certain endomorphisms. This is the content of Theorem \ref{thm:F-is-1-1}.  Let us prepare its formulation.

\begin{definition}
\label{def:endomorphisms-S_0}
 The \emph{strong symmetric semigroup} $sS_\infty$ is the subsemigroup of \linebreak $\End(\ZZ_2^{* \infty})$ generated by finite identifications of letters, i.e. for any $n \in \NN$ and any choice of indices $i(1), \dotsc, i(n)$ the maps
  \[\begin{cases}
     a_k \mapsto a_{i(k)}       & 1 \leq k \leq n  \\
     a_k \mapsto a_k           & k > n 
     \end{cases}
  \]
are in $sS_\infty$.
\end{definition}

We give a refinement of Lemma \ref{LemFCGruppe}.

\begin{lemma}
\label{lem:F-defines-invariant-group}  
If $\cC$ is a group-theoretical category of partitions, then $F(\cC)$ is an $sS_\infty$-invariant normal subgroup of $\ZZ_2^{*\infty}$.  
So $F$ is a well-defined map from group-theoretical categories of partitions to $sS_\infty$-invariant normal subgroups of $\ZZ_2^{*\infty}$.  Moreover, $F$ preserves inclusions and  hence it is a lattice homomorphism.
\end{lemma}
\begin{proof}
Using Lemma \ref{LemFCGruppe} it remains to show that $F(\cC)$ is invariant under finite identification of letters.  Let $g = w(p,l)$ be an element in $F(\cC)$ constructed from a partition $p \in \cC$.  It is clear that we can change a letter in $g$, if the new letter did not appear in $g$ before -- this simply corresponds to $w(p,l')$ with a different labelling $l'$.  If the new letter already appeared in $g$, we connect two blocks of $p$ using Lemma~\ref{LemBlockVerbinden}.  
\end{proof}

The preceding lemma specifies that we can associate an $sS_\infty$-invariant normal subgroup $F(\cC)$ of $\ZZ_2^{*\infty}$ to any group-theoretical category of partitions $\cC$ -- but we can also go back. In fact, \emph{every} $sS_\infty$-invariant normal subgroup of $\ZZ_2^{*\infty}$ comes from such a category. 
This is worked out in the sequel.

\begin{lemma}
For any $sS_\infty$-invariant normal subgroup $H$ of $\ZZ_2^{*\infty}$, the set 
\[\cC_H:=w^{-1} ( H)=\{p\in P \;|\; \textnormal{there is a labelling $l$ such that } w(p,l)\in H\}\subset P\]
 is a group-theoretical category of partitions.
\end{lemma}
\begin{proof}
The pair partition $\sqcap$, the unit partition $\idpart$, the four block partition $\vierpart$, and the pair positioner partition $\primarypart$ are all in $\cC_H$, since they are mapped to the neutral element $e\in H$ for any labelling $l$. 

Let $p$ and $q$ be partitions in $\cC_H$ and denote by $g:=w(p,l)$ and $h:=w(q,l')$ some corresponding elements in $H$ for some labellings $l$ and $l'$. Since $H$ is invariant under permutation of letters, we can assume that the labellings $l$ and $l'$ are chosen in such a way that $g$ and $h$ do not share any letter. If $p$ has no upper points, then $gh=w(p\otimes q,ll')$ is in $H$ and hence $p\otimes q\in \CC_H$. Otherwise, we may write the element $g$ as $g=g_1g_2$, where $g_1$ corresponds to the labelling of the lower points of $p$, and $g_2$ to the upper points of $p$.
The labelling $l''$ of $p\otimes q$ (where $p$ is labelled by $l$ and $q$ by $l'$) is hence of the following form:

\begin{center}
\setlength{\unitlength}{0.5cm}
\begin{picture}(10,6)
 \put(0,2.2){$p\otimes q =$}
 \put(3,1.5){\line(1,0){3}}
 \put(4,0.2){$g_1$}
 \put(4,.8){$\rightarrow$}
 \put(7,1.5){\line(1,0){3}}
 \put(8,0.2){$h_1$}
 \put(8,.8){$\rightarrow$}
 \put(3,3.5){\line(1,0){3}}
 \put(4,4.5){$g_2$}
 \put(4,3.8){$\leftarrow$}
 \put(7,3.5){\line(1,0){3}}
 \put(8,4.5){$h_2$}
 \put(8,3.8){$\leftarrow$}
 \put(4,2.2){$p$}
 \put(8,2.2){$q$}
\end{picture}
\end{center}

As $H$ is closed under conjugation, the element $g_1hg_1^{-1}$ is in $H$. We infer that $w(p\otimes q,l'')=g_1hg_2=g_1hg_1^{-1}g$ is in $H$. Hence, $p\otimes q\in \cC_H$, and $\cC_H$ is closed under tensor products.

The set $\cC_H$ is also closed under involution, since for $p\in \cC_H$ with $w(p,l)=g\in H$, we have $w(p^*,l^*)=g^{-1}\in H$, where $l^*$ denotes the labelling $l$ in reverse order.  It is also closed under rotation, since a partition obtained from another by moving points (from above to below or the converse) at the right hand side yields the same labelled word. If $p'$ is obtained from $p$ by moving the upper left point to below, the element $w(p',l)$ coincides with $aw(p,l)a$, where $a$ is the last letter of the non-reduced word $w(p,l)$. Likewise for rotation of a lower left point to the upper row of points.

It remains to show that $\cC_H$ is closed under the composition of partitions. We first show that $\cC_H$ is closed under composition with a partition of the form:

\begin{center}
\setlength{\unitlength}{0.5cm}
\begin{picture}(10,2)
  \put(0,0){\line(0,1){2}}
  \put(0.8,0.8){$\dotsm$}
  \put(2,0){\line(0,1){2}}
  \put(3,1){\line(0,1){1}}
  \put(4,1){\line(0,1){1}}
  \put(3,1){\line(1,0){1}}
  \put(5,0){\line(0,1){2}}
  \put(5.8,0.8){$\dotsm$}
  \put(7,0){\line(0,1){2}}
\end{picture}
\end{center}

Let $p\in\cC_H$ be a partition on $k$ upper points and $m$ lower points and consider the partition $\idpart^{\otimes i-1}\otimes\baarpart\otimes\idpart^{\otimes m-i-1}$ on $m$ upper points and $m-2$ lower points, where $\sqcup$ connects the $i$-th and the $(i + 1)$-st point. Denote their composition by $p'\in P(k,m-2)$. There is a labelling $l$ such that $g:=w(p,l)$ is in $H$.  For a suitable labelling $l'$, the element $w(p',l')$ arises from $g$ by identifying the $i$-th and the $(i+1)$-st letter.  Since $H$ is invariant under this operation, the partition $p'$ is in $\cC_H$.

It remains to reduce the composition of arbitrary partitions to the previous case. Let $p\in\cC_H$ be a partition on $k$ upper and $l$ lower points, and let $q\in \cC_H$ be on $l$ upper and $m$ lower points. Write $p'$ and $q'$ for the partitions arising from $p$ and $q$, respectively, by rotating their upper points to the right of the lower points. Then $p'$ and $q'$ are both in $\cC_H$.  Composing $q'\otimes p'$ with the partition $\idpart^{\otimes m}\otimes s\otimes \idpart^{\otimes k}$, where $s$ is the partition obtained from $l$ pairs nested into each other,
yields a partition $p''\in\cC_H$ on $k+m$ points. Rotating $k$ points on the right of $p''$ to the upper line gives the composition $qp$ of $p$ and $q$, which hence is in $\cC_H$. Note that in the following picture, the partitions $p'$ and $q'$ consist only of upper points -- and hence the lines connecting the points are above the points.

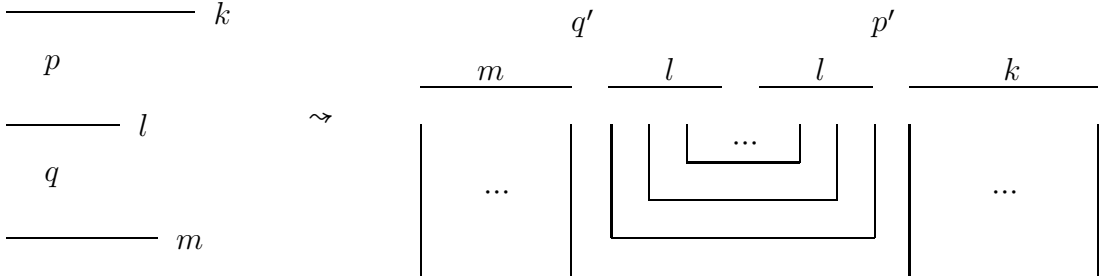
\begin{center}
\setlength{\unitlength}{0.5cm}
\begin{picture}(30,8)
 \put(0,7){\line(1,0){5}}
 \put(5.5,6.7){$k$} 
 \put(1,5.5){$p$}
 \put(0,4){\line(1,0){3}}
 \put(3.5,3.7){$l$}
 \put(1,2.5){$q$}
 \put(0,1){\line(1,0){4}}
 \put(4.5,0.7){$m$}

 \put(8,4){$\leadsto$}

 \put(11,5){\line(1,0){4}}
 \put(12.5,5.2){$m$}
 \put(16,5){\line(1,0){3}}
 \put(17.5,5.2){$l$}
 \put(20,5){\line(1,0){3}}
 \put(21.5,5.2){$l$}
 \put(24,5){\line(1,0){5}}
 \put(26.5,5.2){$k$}
 \put(15,6.5){$q'$}
 \put(23,6.5){$p'$}

 \put(11,0){\line(0,1){4}}
 \put(12.7,2){$\dotsm$}
 \put(15,0){\line(0,1){4}}
 \put(14.8,5){\partii{3}{1}{8}}
 \put(14.8,5){\partii{2}{2}{7}}
 \put(14.8,5){\partii{1}{3}{6}}
 \put(19.3,3.3){$\dotsm$}
 \put(24,0){\line(0,1){4}}
 \put(26.2,2){$\dotsm$}
 \put(29,0){\line(0,1){4}}
\end{picture}
\end{center}

We conclude that $\cC_H$ is closed under the category operations, hence it is a category of partitions, containing $\primarypart$.
\end{proof}

We show now that the map $H \mapsto \cC_H$ is the inverse of $F$.

\begin{theorem}
\label{thm:F-is-1-1}
The maps $F$ and $H \mapsto \cC_H$ are inverse to each other. Hence, the map $F$ is a bijective lattice homomorphism  from group-theoretical categories of partitions onto $sS_\infty$-invariant normal subgroups of $\ZZ_2^{*\infty}$. 
\end{theorem}
\begin{proof}
Firstly, let $H$ be an $sS_\infty$-invariant normal subgroup of $\ZZ_2^{*\infty}$, and let $x \in H$.  Denote by $p$ the partition connecting the letters of the word $x$ if and only if they coincide, and let $l$ be the labelling such that $w(p,l) = x$.  Thus, $p \in \cC_H$ and hence $x \in F(\cC_H)$. (Recall that $x\in F(\cC)$ if and only if $x=w(p,l)$ for some $p\in \cC$ and some labelling $l$.)  Conversely, let $x = w(p,l) \in F(\cC_H)$ where $p \in \cC_H$. By definition, there is a labelling $l'$ such that $w(p,l')\in H$. Now, $H$ is invariant under permutation of letters, thus $x=w(p,l)=w(p,l')\in H$. We deduce that $H = F(\cC_H)$.

Secondly, let $\cC$ be a group-theoretical category of partitions, and let $p \in \cC$. Then $w(p,l) \in F(\cC)$ for any labelling $l$, and hence $p \in \cC_{F(\cC)}$. On the other hand, for $p \in \cC_{F(\cC)}$ there is a labelling $l$ such that $w(p,l) \in F(\cC)$. Thus, $w(p,l) = w(q,l')$ for some partition $q \in \cC$ and some labelling $l'$. 
By Remark \ref{RemSL}, we have $w(p,l)=w(\tilde p,\tilde l)$ where $\tilde p$ is the single leg version of $p$. Likewise  $w(q,l')=w(\tilde q,\tilde l')$. Now, $w(\tilde p,\tilde l)=w(\tilde q,\tilde l')$ implies that the partitions $\tilde p$ and $\tilde q$ coincide, since they are in single leg form (the labellings by $\tilde l$ and $\tilde l'$ yield reduced words, respectively). Thus, $p\in\CC$ by Lemma \ref{LemmaX}.  
This finishes the proof of $\cC = \cC_{F(\cC)}$.
\end{proof}

We conclude that there is a one-to-one correspondence between group-theoretical categories of partitions and $sS_\infty$-invariant normal subgroups of $\Z_2^{*\infty}$. 
The map $F$ hence transfers the problem of classifying all group-theoretical categories of partitions to a problem in group theory.
In \cite{RaumWeberSemiDirect} quotients by these subgroups play an important role as well as the following finite versions of $F$. Note that in \cite{RaumWeberSemiDirect} the fact that $F$ is injective is derived by different methods, namely as a consequence of algebraic considerations.

If $n\in \N$ is a natural number and $\CC$ is a group-theoretical category of partitions, we denote by $F_n(\CC)$ the subset of $\Z_2^{*n}$ formed by elements $w(p,l)$ where $p\in\CC$ and $l$ runs through all possible labellings of $p$ with letters $a_1,\ldots,a_n$. For technical reasons, we only consider partitions in $\CC$ with at most $n$ blocks such that all blocks are labelled by mutually different letters. Owing to Lemma \ref{LemBlockVerbinden}, we could equivalently label all points with any letters from $a_1,\ldots,a_n$ without taking into account that different blocks shall be labelled by different letters. Thus, $F_n(\CC)$ does not contain full information about $\CC$ since $F_n$ is not injective -- but $F_\infty(\CC):=F(\CC)$ does. 

\begin{definition}
Let $sS_n$ be the subsemigroup of $\textnormal{End}(\Z_2^{*n})$ generated by identifications of letters.
A \emph{strongly symmetric reflection group} $G$ is the quotient of $\Z_2^{*n}$ by an $sS_n$-invariant normal subgroup of $\Z_2^{*n}$.
\end{definition}

\begin{example}\label{ExFnC}
 We give examples of correspondences under the map $F_n$, for\linebreak  $n\in\N\cup\{\infty\}$.  
 \begin{itemize}
 \item[(a)] If $\CC=\langle\primarypart\rangle$, then $F_n(\CC)=\{e\}\subset\Z_2^{*n}$, since $F_n(\CC)$ is generated by $w(\primarypart,l)=e$. Hence, the strongly symmetric reflection group $\ZZ_2^{*\infty}/F_n(\CC)$ associated to \linebreak$\langle\primarypart\rangle$  is $\Z_2^{*n}$.
 \item[(b)] The subgroup $F_n(\CC)\subset\Z_2^{*n}$ for $\CC=\langle\halflibpart,\vierpart\rangle$ is generated by all elements $a_{i(1)}a_{i(2)}a_{i(3)}a_{i(1)}a_{i(2)}a_{i(3)}$. The corresponding strongly symmetric reflection group is the quotient of $\Z_2^{*n}$ by the relations $abc=cba$ for all $a,b,c$.
 \item[(c)] The category $\langle h_s,\vierpart\rangle$ for $s\geq 2$ corresponds to the strongly symmetric reflection group in which all elements satisfy $(ab)^s=e$. Note that for $s=2$ this group is abelian, namely it is $\Z_2^{\oplus n}$. (We have  $\langle h_2,\vierpart\rangle=\langle\crosspart,\vierpart\rangle$.)
 \item[(d)] The reflection group corresponding to the category $\langle \halflibpart, h_s, \vierpart\rangle$, $s\geq 3$ is the quotient of $\ZZ_2^{*n}$ by the relations $(ab)^s=e$ and $abc=cba$.  
 \item[(e)] The category  $\langle\crosspart,\singleton\otimes\singleton,\vierpart\rangle$ 
 corresponds to the reflection group $\Z_2$, since $F_n(\CC)$ is generated by $a_ia_j=w(\singleton\otimes\singleton,(a_i,a_j))$. 
 \item[(f)] The category  $\langle\crosspart,\singleton,\vierpart\rangle$ corresponds to the trivial reflection group $\{e\}$, since $F_n(\CC)$ contains all $a_i=w(\singleton,(a_i))$. 
 \end{itemize}
\end{example}

\section{Consequences for easy quantum groups}\label{SectConsequences}

Woronowicz defined the notion of a \emph{compact matrix quantum group} in \cite{Woronowicz87}, see also \cite{Woronowicz91}. For $n\in\N$, it is given by
 \begin{itemize}
  \item a unital $C^*$-algebra $A$,
  \item elements $u_{ij}\in A$, $1\leq i,j\leq n$ that generate $A$ as a $C^*$-algebra,
  \item and a *-homomorphism $\Delta: A\to A\otimes_{\mini} A$, mapping $u_{ij}\mapsto \sum_{k=1}^n u_{ik}\otimes u_{kj}$.
  \item Furthermore, we require the $n\times n$-matrices $u=(u_{ij})$ and $u^t=(u_{ji})$ to be invertible.
 \end{itemize}
We also write $C(G)$ for the C$^*$-algebra $A$ and refer to $G$ as the quantum group.
In \cite{Woronowicz88} Woronowicz proved a Tannaka-Krein duality for compact matrix quantum groups, stating that they are determined by their intertwiner spaces.
These are given by:
\[\Hom(k,l)=\{T:(\C^n)^{\otimes k}\to (\C^n)^{\otimes l} \textnormal{ linear}\;|\; Tu^{\otimes k}=u^{\otimes l}T\}, \qquad k,l\in\N_0\]

Banica and Speicher introduced \cite{BanicaSpeicher09} a particular class of compact matrix quantum groups whose intertwiner spaces are given by the combinatorics of partitions.
Given a partition $p \in P(k,l)$ and two multi-indices $(i_1, \dotsc, i_k)$, $(j_1, \dotsc, j_l)$, we label the diagram of $p$ with these numbers (now, the upper and the lower row both are labelled from left to right) and we put:
\[\delta_p(i,j)
  =
  \begin{cases}
    1 & \text{if } p \text{ connects only equal indices,} \\
    0 & \text{otherwise} 
  \end{cases}\]
For every $n \in \NN$, there is a map $T_p: (\C^n)^{\ot k} \to (\C^n)^{\ot l}$  associated with $p$, which is given by:
\[T_p(e_{i_1} \ot \dotsm \ot e_{i_k}) = \sum_{1 \leq j_1, \dotsc, j_l \leq n} \delta_p(i, j) e_{j_1} \ot \dotsm \ot e_{j_l}  \]

\begin{definition}[Definition 6.1 of \cite{BanicaSpeicher09} or Definition 2.1 of \cite{BanicaCurranSpeicher10}]
  A compact matrix quantum group $S_n\subset G\subset O_n^+$ is called \emph{easy}, if there is a category of partitions $\cC$ given by $\CC(k,l) \subset P(k,l)$, for all $k,l \in \NN_0$ such that:
\[\Hom( k, l) = \textnormal{span} \{T_p \;|\; p \in \CC(k,l) \} \]
\end{definition}
In this sense, easy quantum groups could also be called \emph{partition quantum groups}.

Based on Woronowicz' Tannaka-Krein duality we obtain the following theorem, which is the basis of all combinatorial investigation on easy quantum groups.

\begin{theorem}[\cite{BanicaSpeicher09}]
\label{thm:classification-easy-quantum-groups-by-partitions}
  There is a bijection between categories of partitions and series of easy quantum groups (i.e. a category $\CC$ corresponds to a series consisting of quantum groups $G_n$ for all dimensions $n\in\N$) up to similarity.
\end{theorem}

Thus, easy quantum groups are completely determined by their categories of partitions. We refer to \cite{BanicaSpeicher09}, \cite{BanicaCurranSpeicher10}, \cite{Weber13} or \cite{RaumWeberVolleKlassif} for more details on easy quantum groups.
We say that an easy quantum group is \emph{group-theoretical} if it corresponds to a group-theoretical category of partitions.
Note that the maximal group-theoretical easy quantum group $H_n^{[\infty]}$ corresponds to the category $\langle\primarypart\rangle$, where $H_n^{[s]}$ is the hyperoctahedral series as defined in \cite{BanicaCurranSpeicher10} (see also Remark \ref{RemKL}).

In \cite{RaumWeberSemiDirect}, we develop another approach to group-theoretical easy quantum groups. 
The main theorem there is the following.

\begin{theorem}[{\cite[Theorem 3.1]{RaumWeberSemiDirect}}]
 Let $S_n\subset G\subset O_n^+$ be an easy quantum group with associated category $\CC$ of partitions. If $\CC$ is group-theoretical, then $G$ may be written as a semi-direct product:
 \[C(G)\cong C^*(\Z_2^{*n}/F_n(\mathcal C)) \Join C(S_n)\]
\end{theorem}

This reveals a lot of information about the quantum group $G$. 
See \cite{RaumWeberSemiDirect} for consequences of this picture. 
The main difference of the two approaches of the present article and \cite{RaumWeberSemiDirect} is how to prove that the map $F$ of Theorem \ref{thm:F-is-1-1} is injective. In the present article, this is obtained by purely combinatorial means whereas in \cite{RaumWeberSemiDirect} it is derived as a consequence of the above theorem. Nevertheless, the results of \cite[Sect. 5.1]{RaumWeberSemiDirect} can also be deduced from the present article. To any variety of  groups, we may associate an $sS_\infty$-invariant normal subgroup of $\ZZ_2^{*\infty}$. Since there are uncountably many varieties of groups and since $F$ is bijective, we infer that there are uncountably many easy quantum groups. See \cite[Sect. 5.1]{RaumWeberSemiDirect} for details.

\section{Associated C$^*$-algebras}\label{SectAssocCStar}

In this section we briefly study the C$^*$-algebras associated to group-theoretical easy quantum groups. Since these quantum groups in turn correspond to group-theoretical categories of partitions, let us speak of a C$^*$-algebra $A_\CC(n)$ associated to a category $\CC$.

\begin{proposition}\label{GestaltAprimary}
The C$^*$-algebra $A_\CC(n)=C(H_n^{[\infty]})$  associated to the category $\CC=\langle\primarypart\rangle$ is the universal C$^*$-algebras generated by elements $u_{ij}$, $i,j,=1,\ldots,n$ such that
\begin{enumerate}
 \item the $u_{ij}$ are local symmetries (i.e. $u_{ij}=u_{ij}^*$ and $u_{ij}^2$ is a projection),
 \item the projections $u_{ij}^2$ fulfill $\sum_k u_{ik}^2=\sum_ku_{kj}^2=1$ for all $i,j$,
 \item and we have $u_{ij}^2u_{kl}=u_{kl}u_{ij}^2$ for all $i,j,k,l$.
\end{enumerate}
\end{proposition}
\begin{proof} 
Since $\langle\vierpart\rangle$ is contained in $\langle\primarypart\rangle$, the C$^*$-algebra $A_\CC(n)$ fulfills the relations  $u_{ik}u_{jk}=u_{ki}u_{kj}=0$ whenever $i\neq j$ and $\sum_k u_{ik}^2=\sum_ku_{kj}^2=1$, as well as $u_{ij}=u_{ij}^*$, for all $i,j$ (see \cite{Weber13} for the relations of the C$^*$-algebra associated to the category $\langle\vierpart\rangle$ or rather of the free hyperoctahedral quantum group $H_n^+$).
It follows that $u_{ij}^2$ is a projection, since $u_{ij}^2=u_{ij}^2(\sum_k u_{ik}^2)=\sum_ku_{ij}(u_{ij}u_{ik})u_{ik}=u_{ij}^4$. 
A direct computation verifies that $T_p u^{\otimes 3}=u^{\otimes 3}T_p$ for $p=\primarypart$ is equivalent to (3). (See also \cite{Weber13} or \cite{RaumWeberVolleKlassif} for such computations.)
\end{proof}

This proposition shows that the elements $u_{ij}^2$ fulfill the relations of $S_n$, the (classical) permutation group (or rather of $C(S_n)$). The squares of the elements $u_{ij}$ of the above C$^*$-algebra thus behave like commutative elements. See also the picture of group-theoretical easy quantum groups  as semi-direct products \cite{RaumWeberSemiDirect}.

In the C$^*$-algebra $A_\CC(n)$ associated to a group-theoretical category $\cC$, the relations on the generators may be read directly from the partitions in single leg form. Let $p=a_{i(1)} \dotsm a_{i(l)} \in P(0,l)$ be a partition without upper points in single leg form. We consider $p$ as a word in the letters $a_1, \dotsc ,a_k$ (labelled from left to right). If we replace the letters $a_i$, $1 \leq i \leq k$, in $p$ by some choice of generators $u_{ij}$, $1 \leq i,j \leq n$, we obtain an element in $\AKat$; replacing the letters by the according elements $u_{ij}^2$ yields a projection $q \in \AKat$.

\begin{proposition}
\label{PropWordPartialIsom}
 Let $\mathcal C$ be a group-theoretical category and let $p = a_{i(1)} \dotsm a_{i(l)}$ be a partition in single leg form.
 The following assertions are equivalent:
\begin{enumerate}
 \item $p \in \mathcal C$.
 \item $a_{i(1)} \dotsm \, a_{i(l)} = q$ in $\AKat$ for all choices $a_r \in \{u_{ij} \; | \; i,j = 1, \dotsc, n\}$, $1 \leq r \leq k$, where $q$ is the according range projection.
 \item For some $1 \leq s \leq l$, we have $q a_{i(1)}\dotsm \, a_{i(s)} = q a_{i(l)} \dotsm a_{i(s+1)}$ in $\AKat$ for all choices $a_r \in \{u_{ij} \; | \; i,j = 1, \dotsc, n\}$, $1 \leq r \leq k$, where $q$ is the according range projection.
\end{enumerate}
\end{proposition}
\begin{proof} 
The linear map $T_p: \C \to (\C^n)^{\otimes k}$ associated with $p$ is given by
\[
  T_p(1)
  =
  \sum_{i(1), \dotsc, i(k) = 1}^n
  {
    \delta_p(i)e_{i(1)} \otimes \dotsm \otimes e_{i(k)}
  }   
\]
We have
\[
  u^{\otimes k} (T_p \otimes 1)(1 \ot 1)
  =
  \sum_{i(1), \dotsc, i(k) = 1}^n
  {
    e_{i(1)} \otimes \dotsm \otimes e_{i(k)} \otimes
    \left( \sum_{j(1), \dotsc, j(k) = 1}^n \delta_p(j) \cdot u_{i(1) j(1)} \dotsm \, u_{i(k)j(k)} \right)
  }
  \]
so that $p \in \mathcal C$, if and only if for all multi-indices $i=(i(1), \dotsc ,i(k))$ the equation:
\[
  \sum_{j(1), \dotsc, j(k) = 1}^n
  {
    \delta_p(j) \cdot u_{i(1) j(1)} \dotsm \, u_{i(k)j(k)}
  }
  =
  \delta_p(i)
\]
holds.  Now, assume (1) and let us show (2).  Make a choice of $a_r \in \{ u_{ij} | i,j = 1, \dotsc, n \}$ for all ${r \in \{1, \dotsc, k\}}$.  Then there are multi-indices $i$ and $j$ satisfying $\delta_p(i) = \delta_p(j) = 1$ such that $a_{i(1)} \dotsm \, a_{i(k)} = u_{i(1)j(1)} \dotsm \, u_{i(k) j(k)}$.  Let $q$ be the projection given by $q := u_{i(1)j(1)}^2 \dotsm \, u_{i(k) j(k)}^2$.  Then, using $u_{ij}u_{ir}=\delta_{jr}u_{ij}^2$ we obtain:
\begin{align*}
  u_{i(1)j(1)} \dotsm u_{i(k)j(k)}
  & = 
  u_{i(1)j(1)} \dotsm u_{i(k)j(k)} \left( \sum_{r(1), \dotsc, r(k) = 1}^n \delta_p(r) \cdot u_{i(k) r(k)} \dotsm \, u_{i(1)r(1)} \right) \\
  & =
  \sum_{r(1), \dotsc, r(k) = 1}^n
    \delta_p(r) \cdot u_{i(1)j(1)} \dotsm \, u_{i(k -1)j(k - 1)} \\
    & \qquad \qquad
    (\delta_{j(k) r(k)} u_{i(k)j(k)}^2)
    u_{i(k -1)r(k - 1)} \dotsm u_{i(1)r(1)}
   \\
  & = 
  \sum_{r(1), \dotsc, r(k) = 1}^n \delta_p(r) \cdot \delta_{i(1) r(1)} u_{i(1)j(1)}^2 \dotsm \, \delta_{i(k)r(k)} u_{i(k) j(k)}^2 \\
  & =
  q
\end{align*}
This proves (2). 
Conversely, assume (2) and let $i$ be any multi-index.  If $\delta_p(i) = 0$, then $u_{i(1)j(1)} \dotsm \, u_{i(k)j(k)} = 0$ for any multi-index $j$ that satisfies $\delta_p(j) = 1$, since in this product there are at least two local symmetries that have mutually orthogonal support in the centre of $\AKat$.  Hence:
\[
  \sum_{j(1), \dotsc, j(k) = 1}^n \delta_p(j) \cdot u_{i(1)j(1)} \dotsm \, u_{i(k)j(k)}
  =
  0
   \]
Similarly if $\delta_p(i) = 1$, then, using the assumption (2):
\[
  u_{i(1)j(1)}^2 \dotsm \, u_{i(k)j(k)}^2
  =
  \begin{cases}
    u_{i(1)j(1)} \dotsm \, u_{i(k)j(k)}  , & \text{if } \delta_p(j) = 1  , \\
    0  , & \text{otherwise}
  \end{cases}
\]
We obtain that:
\[
  \sum_{j(1), \dotsc, j(k) = 1}^n \delta_p(j) \cdot u_{i(1)j(1)} \dotsm \, u_{i(k)j(k)}
  =
  \sum_{j(1), \dotsc, j(k) = 1}^n u_{i(1)j(1)}^2 \dotsm \, u_{i(k)j(k)}^2
  =
  1
 \]
This proves (1).  The assertions (2) and (3) are equivalent, since all projections $a_r^2$, $1 \leq r \leq m$ are absorbed by $q$ and $qa_{i(1)} \dotsm \, a_{i(k)} = a_{i(1)} \dotsm \, a_{i(k)}$.
\end{proof}

\end{document}